\documentclass[12pt,a4paper]{amsart}
\usepackage[utf8]{inputenc}
\usepackage[body={6.5in, 8.5in},left=1.3in,right=1.3in]{geometry}

\usepackage[utf8]{inputenc}
\usepackage[english]{babel}
\usepackage{amsmath}
\usepackage{amsthm}
\usepackage[shortlabels]{enumitem}
\usepackage{amssymb}
\usepackage{amsfonts}
\usepackage{mathrsfs}
\usepackage{setspace}
\usepackage{accents}
\usepackage[dvipsnames]{xcolor}
\usepackage{mathtools}
\usepackage{bbm}
\usepackage{tikz}
\usepackage{tikz-cd}
\usepackage{quiver}
\usepackage{xfrac}
\usepackage{faktor}
\usepackage{marvosym}
\usepackage{bm} 
\usepackage[normalem]{ulem}
\usepackage{hyperref}
\hypersetup{
    colorlinks=true,
    linkcolor=blue,
    filecolor=magenta,      
    urlcolor=cyan,
    pdftitle={Overleaf Example},
    pdfpagemode=FullScreen,
    }

\newtheorem{lem}{Lemma}[section]
\newtheorem{thm}[lem]{Theorem}
\newtheorem*{thm*}{Theorem}
\newtheorem{cor}[lem]{Corollary}
\newtheorem{prop}[lem]{Proposition}
\newtheorem*{prop*}{Proposition}

\newtheorem{question}[lem]{Question}
\newtheorem*{question*}{Question}

\theoremstyle{definition}
\newtheorem{df}[lem]{Definition}
\newtheorem{rmk}[lem]{Remark}

\newtheorem{egs}[lem]{Examples}
\newtheorem*{egs*}{Examples}
\newtheorem{eg}[lem]{Example}

\theoremstyle{remark}
\newtheorem*{remarks*}{Remarks}

\newtheorem*{note*}{Note}

\newcommand{\hk}{\textrm{hyper-K\"ahler}}
\newcommand{\hks}{hyper-K\"ahlers}

\newcommand{\simga}{\sigma}

\newcommand{\N}{\mathbb{N}}

\newcommand{\K}{\mathbb{K}}
\newcommand{\Z}{\mathbb{Z}}
\newcommand{\Q}{\mathbb{Q}}
\newcommand{\R}{\mathbb{R}}
\newcommand{\C}{\mathbb{C}}

\newcommand{\proj}{\mathbb{P}}

\newcommand{\Ccal}{\mathcal{C}}
\newcommand{\Dcal}{\mathcal{D}}

\newcommand{\Ocal}{\mathcal{O}}

\DeclareMathOperator{\ord}{\textup{ord}}

\DeclareMathOperator{\Alb}{\textup{Alb}}
\DeclareMathOperator{\alb}{\textup{alb}}
\DeclareMathOperator{\Aut}{\textup{Aut}}
\DeclareMathOperator{\Bir}{\textup{Bir}}
\DeclareMathOperator{\Bran}{\textup{Bran}}

\DeclareMathOperator{\Exc}{\textup{Exc}}
\DeclareMathOperator{\Gal}{\textup{Gal}}
\DeclareMathOperator{\GL}{\textup{GL}}
\DeclareMathOperator{\Gr}{\textup{Gr}}

\DeclareMathOperator{\Hilb}{\textrm{Hilb}}
\DeclareMathOperator{\Ker}{\textup{Ker}}

\DeclareMathOperator{\Km}{\textup{Km}}
\DeclareMathOperator{\id}{\textup{id}}
\DeclareMathOperator{\Jac}{\textrm{Jac}}

\DeclareMathOperator{\Mon}{\textup{Mon}}

\DeclareMathOperator{\NS}{\textup{NS}}

\DeclareMathOperator{\Ram}{\textup{Ram}}
\DeclareMathOperator{\rk}{\textup{rk}}
\DeclareMathOperator{\Sym}{\textup{Sym}}

\DeclareMathOperator{\supp}{\textup{supp}}

\newcommand{\simX}{\widetilde{X}}
\newcommand{\simY}{\widetilde{Y}}

\newcommand{\simf}{\widetilde{f}}

\newcommand{\simvarphi}{\widetilde{\varphi}}

\newcommand{\simnu}{\widetilde{\nu}}

\title{Galois covers of Calabi-Yau manifolds}
\author{Matteo Verni}
\date{\today}

\begin{document}
\begin{abstract}
     We study Galois rational maps between smooth projective varieties with trivial canonical bundle, with a particular interest in the case where the codomain is \hk. We obtain results about the birational geometry and the Galois closure of rational maps between Calabi-Yau manifolds.
     We apply these general results to a number of well-known examples.
 \end{abstract}
\maketitle
\tableofcontents
\section{Introduction}

A \textit{projective Calabi-Yau manifold} is a smooth projective complex algebraic variety with trivial canonical bundle. 
In the study of these varieties it is natural to consider nontrivial dominant maps between them.
However, dominant morphisms between Calabi-Yau manifolds of the same dimension do not ramify, hence they are étale, which is a severe restriction.
On the other hand, there exist many interesting dominant generically finite \textit{rational} maps between projective Calabi-Yau manifolds; the overarching theme of the present work is the study of Calabi-Yau manifolds and the rational maps between them.

Given a generically finite dominant rational map of complex algebraic varieties \(f\colon Y\dashrightarrow X\), its \textit{monodromy group} is the image of the fundamental group of a suitable Zariski neighbourhood of a general point \(x\) in the permutation group of the fiber above \(x\). 
Another way to understand this group is as follows (\cite{Har79}): if \(\K \) denotes the Galois closure of \(f^* \colon \C(X)\hookrightarrow \C(Y)\), then 
\[\Mon(f)=\Gal(\K /\C(X)).\]

In this article, we study the Galois theory of dominant rational maps between Calabi-Yau manifolds of the same dimension. For convenience, we introduce the following notation:
\begin{df}
Let \(X,Y\) be two algebraic varieties of the same dimension.
A \textit{rational cover} is a dominant rational map \(f\colon Y\dashrightarrow X\). We moreover say that \(f\) is a \textit{Calabi-Yau rational cover} when \(Y\) is a projective Calabi-Yau manifold.
    %(resp. abelian, \hk )

A \textit{Galois cover} is a rational cover \(f \colon Y\dashrightarrow X\) such that the finite extension \(f^* \colon \C(X)\hookrightarrow \C(Y)\) is Galois. We denote the Galois group of this extension by 
\(\Gal(f)\), or by \(\Gal(Y/X)\) when there is no risk of ambiguity.
   
\end{df}
 If \(Y\) is abelian (resp. \hk), we say that \(f\) is an abelian (resp. \hk) rational cover. In this paper we will study the case where \(X\) is Calabi-Yau as well, and most often \hk.

In terms of monodromy, a rational cover \(f \colon Y \dashrightarrow X\) is Galois when its monodromy group coincides with the group of birational automorphisms of the cover, i.e. when 
\[\Mon (f)=\Bir(f):=\{g\in \Bir(Y)\; f\circ g=f\}.\] 
When \(f\) is Galois,  we can write \(X\) up to birational equivalence as a finite quotient of \(Y\): there exists an open subvariety \(U \subset Y\) on which the finite group \(G:=\Gal(f)\) acts by automorphisms and such that
\[   U/G \sim_{bir}X .\]

\begin{egs}\label{egs_examples_basic}
\hfill 

\(\bullet\) The simplest example of Calabi-Yau Galois cover of a Calabi-Yau manifold is the classical \textit{Kummer construction}, which produces a \(2\) to \(1\) map 
\[A\dashrightarrow \Km(A)\]
where \(\Km(A)\) is a K3 surface and the Galois map is given as the quotient by the action of the natural involution sending \(x\) to \(-x\).

\(\bullet\) An example in dimension \(2n\) for any \(n \geq 2\) is given by the Hilbert scheme of points \(S^{[n]}\) on a K3 surface \(S\). Indeed, there is a CY Galois cover
\[S^n\dashrightarrow S^{[n]}\] 
 obtained by composing the birational inverse of the Hilbert-Chow morphism \(S^{[n]} \rightarrow S^{(n)}\) with the natural quotient map \(S^n \rightarrow S^{(n)}\) under the permuting action of the symmetric group \(S_{n}\).
\end{egs}

Abelian Galois covers are particularly interesting: a pleasant feature of this case is that the birational automorphisms of \(A\) are actually automorphisms. Here are a few simple examples of abelian Galois covers:

\begin{egs}\label{egs_alexeev_abelian_quotients}
\hfill

\(\bullet\) One can combine the above two examples into an abelian cover of \(\Km(A)^{[n]}\) by considering the composition
\[A^n \dashrightarrow \Km(A) ^n\dashrightarrow \Km(A)^{[n]}.\] 

\(\bullet\) We consider another infinite family of \hk{} manifolds, namely that of \textit{generalized Kummer} manifolds: for each \(n\geq 1\), \( \Km_n(A)\) is a \(2n\)-dimensional \hk{}  manifold, obtained as the fiber above \(0\) of the morphism
\(A^{[n+1]} \rightarrow A\)
\[Z \rightarrow \sum_{z\in |Z|} \ell_\C (\, \Ocal_{Z,z})\cdot z \]
given by the sum map,
where \(A\) is an abelian surface and \(|Z|\) denotes the support of the \(0\)-dimensional subscheme \(Z\subset A\) of length \(n+1\). 
The manifold \(\Km_n(A)\) admits a rational map

\[A^n \dashrightarrow \Km_n(A)\]
defined on general tuples as 
\[(x_1,\dots , x_n)\rightarrow \left(x_1,\dots , x_n, \sum_{i=1}^n -x_i \right), \]
and whose Galois group is \(S_{n+1}\) (\cite{Ale2022}). In particular, one recovers the classical Kummer surface construction for \(n=1\).

\end{egs}

%\end{egs*}

%The question of what can be said about the birational geometry of a finite birational quotient of a Calabi-Yau manifold has been already brought up in the literature: for example in \cite{kol_lar2009} the authors characterize the group actions for which the quotient is uniruled or rationally connected; in \cite{Gac2024}, the author proves that the 
\subsection{Motivation}

The first question we study in this article is the following:
\begin{question}\label{quest_Galoiscover}
Does a given CY manifold \(X\) admit nontrivial Calabi-Yau Galois covers \(Y\dashrightarrow X\)?
\end{question}

This general problem is inspired by an intriguing question relating to \hk{} geometry, recently posed by Laza in a series of talks:

\begin{question}[Laza]\label{quest_laza}
Does every deformation family of \hk{} manifolds admit one member which has an abelian Galois cover?
\end{question}
Note that this question has an easy positive answer for the deformation families of K3\(^{[n]}\) and \(\Km_n\)-type, by Examples \ref{egs_alexeev_abelian_quotients}. On the other hand, for the OG6 and OG10 deformation class Question \ref{quest_laza} is much more subtle. 
Were it to have positive answer in general, it would shed light on possible new families of \hk{} manifolds.

In fact, Examples \ref{egs_alexeev_abelian_quotients} can be seen as instances of a more general construction: in \cite{Ale2022}, it is pointed out that one can produce many singular \hks \ as quotients of a power of an abelian surface by a certain Weyl group. The idea that such singular \hk{} varieties may admit smooth birational models is part of the motivation behind Laza's Question.

Question \ref{quest_Galoiscover} is also motivated by \cite{kol_lar2009}, where the authors study a problem that is in some sense inverse to Question \ref{quest_Galoiscover}; they start from a mildly singular Calabi-Yau variety \(X\) and study the birational geometry of a finite quotient \(X/G\), like its Kodaira dimension and rational connectedness.
In another direction, the paper \cite{Gac2024} disproves the existence in dimension 4 of nontrivial quotients of strict Calabi-Yau manifolds under an action that is free in codimension 2 of a finite group.

It is important to observe that not all CY rational covers are Galois. The simplest example of this is given in \cite[Remark 15.1]{kol_lar2009}, after an observation by De-Qi Zhang, and constructed as follows. Let \(A\) be an abelian surface and consider the diagram
% https://q.uiver.app/#q=WzAsNCxbMCwxLCJLKEEpIl0sWzEsMSwiSyhBKSJdLFswLDAsIkEiXSxbMSwwLCJBIl0sWzAsMSwiXFxvdmVybGluZXttfSIsMCx7InN0eWxlIjp7ImJvZHkiOnsibmFtZSI6ImRhc2hlZCJ9fX1dLFsyLDAsIiIsMCx7InN0eWxlIjp7ImJvZHkiOnsibmFtZSI6ImRhc2hlZCJ9fX1dLFszLDEsIiIsMix7InN0eWxlIjp7ImJvZHkiOnsibmFtZSI6ImRhc2hlZCJ9fX1dLFsyLDMsIm0iXV0=
\begin{equation}\label{ex_kollar_larsen}
\begin{tikzcd}
	A & A \\
	{\Km(A)} & {\Km(A)}
	\arrow["\mu_m", from=1-1, to=1-2]
	\arrow["\pi", dashed, from=1-1, to=2-1]
	\arrow["\pi", dashed, from=1-2, to=2-2]
	\arrow["\psi", dashed, from=2-1, to=2-2]
\end{tikzcd},
\end{equation}
where \(\pi\) is the natural Galois cover onto the Kummer surface, \(\mu_m\) is the multiplication by \(m\in \Z\) and \(\psi\) the induced self-map on the Kummer surface \(\Km(A)\). The map \(\pi\circ \mu_m =\psi \circ \pi \) is Galois and its Galois group is the group of affine transformations
\[(g,t)\in \{\id , \iota \}\ltimes A[m],\]
which act by \((g,t)(x)=g(x) +t\),
where \(A[m]\) denotes the \(m\)-torsion subgroup and \(\iota\) is the involution \(x\mapsto -x\). For \(m>2\),  the subgroup \(\{\id , \iota \}\) is not normal in this semidirect product, which implies that \(\psi\) is not Galois.
However, in the above example it follows by construction that up to composing \(\psi\) with the Calabi-Yau rational cover \(\pi\), we do obtain a Calabi-Yau Galois cover.

\color{black}
This suggests the following definition:

\begin{df}\label{def_Galois_like}
We say that a rational cover \(f \colon X'\dashrightarrow X\) of Calabi-Yau manifolds is a \textit{Calabi-Yau Galois-like cover} if it factors a Calabi-Yau Galois cover, i.e., there exists a CY rational cover \(g\colon Y\dashrightarrow X'\) such that \(f\circ g \colon Y \dashrightarrow X\) is Galois.
\end{df}

The map \(\psi\) above is a typical example of a CY Galois-like cover which is not Galois. One is then led to ask the following, which is the second main question we consider here:
\begin{question}\label{quest_Galois_factors}
Given a Calabi-Yau rational cover \(f \colon X'\dashrightarrow X\), is it a Calabi-Yau Galois-like cover?
\end{question}
More generally, we are interested in geometric criteria that allow to distinguish Calabi-Yau Galois-like covers amongst all Calabi-Yau rational covers.

From the point of view of birational geometry, a more natural formulation of Question \ref{quest_Galois_factors} would be:

\begin{question}\label{quest_birat_Galois_factors}
\textit{Given a rational cover \(X'\dashrightarrow X\) of Calabi-Yau manifolds, does it factor a Galois cover \(Y\dashrightarrow X\) with Kodaira dimension \(\kappa(Y)\) equal to zero?}
\end{question}
%} 

%Note that by picking a `` geometric" Galois closure (i.e. any algebraic variety \(\widehat{Y}\) 
\begin{rmk}\label{rmk_birational_question}

Let \(Z\) be the variety (well-defined up to birational equivalence) such that \(\C(Z)\) is the Galois closure of \(f^* \colon \C(X)\hookrightarrow \C(Y)\). 
Then an equivalent formulation of Question \ref{quest_birat_Galois_factors} is whether \(Z\) has Kodaira dimension equal to zero.
\end{rmk}

Admitting the abundance conjecture for Kodaira dimension zero, the main difference between Question \ref{quest_birat_Galois_factors} and Question \ref{quest_Galois_factors} is that in the former we allow the variety \(Y\) to be a CY \textit{variety} in the sense of \cite{kol_lar2009}, that is, a K-trivial variety with canonical singularities.
%If we assume for \(\widehat{Y}\) the existence of a good minimal model (equivalently, the Abundance Conjecture) it is easily seen that this condition is also sufficient.

As already mentioned, checking that \(f\colon X' \dashrightarrow X\) is not Galois amounts to computing \(\Bir(X'/X)\). To disprove that it is a Calabi-Yau Galois-like cover we need to study first all Calabi-Yau rational covers of \(X'\) and their birational automorphisms over \(X'\).

\begin{rmk}\label{rmk_surfaces}
If \(\dim X =2\), the two questions are equivalent, since any smooth surface \(S\) of Kodaira dimension zero with \(p_g (S)\neq 0\) is birational to a Calabi-Yau surface.
\end{rmk}

We will study Question \ref{quest_Galois_factors} and Question \ref{quest_birat_Galois_factors} for the following examples.

\begin{eg}[Voisin map on \(F_1(W)\)]\label{eg_voisinmap_F_1}
     Let \(W\subset \proj^5\) be a smooth cubic fourfold and let \(F:=F_1(W)\) be the Hilbert scheme of lines lying on \(W\): it is called the \textit{Fano variety of lines on \(W\)}. It is a \hk{} manifold of dimension 4 (\cite{BeaDon1985}). We can construct a rational endomorphism 
\[\nu \colon F\dashrightarrow F\] as follows: if the line \([\ell] \in F\) is general, there is only one projective plane \(\Pi_\ell\) tangent to \(W\) everywhere along \(\ell\). Counting intersections with multiplicity, we have
    \[\Pi_\ell \cap W = 2\ell +\ell'\]
    where \(\ell'\) is another line, so we may set \(\nu([\ell]):=[\ell']\).
    This rational map was introduced in \cite{Voi2004} and it is often called the \textit{Voisin map}. It has degree 16 and has been the object of much investigation (see for example \cite{Ame2009}, \cite{GouKou2023}, \cite{GioGio2024}). 
     
   % It was recently shown (\cite{GioGio2024}) that it has maximal monodromy group, i.e. \[\Mon(\nu)\simeq S_{16},\]a fact which will play a role in Theorem \ref{thm_vosinmap_abeliancover}.
\end{eg}  

\begin{eg}[Voisin map to LLSvS]\label{eg_voisinmap_LLSvS}
In the same context as above, Voisin introduces a rational map of degree 6 \[F\times F \dashrightarrow Z\]

\noindent where \(Z\) is the LLSvS \hk{}  \(8\)-fold that was constructed in \cite{LLSvS2017}, also starting from the cubic \(W\). We will recall the construction of the rational map in Section \ref{sec_uniruled_abovebranch}.

\end{eg}
\begin{eg}[CY fibered in abelian varieties]\label{eg_abelian_fibrations}
    Given a fibration \(\pi: X\rightarrow B\) whose generic fiber is abelian, one can produce rational endomorphisms 
% https://q.uiver.app/#q=WzAsMyxbMCwwLCJYIl0sWzIsMCwiWCJdLFsxLDEsIlxccHJval4xIl0sWzAsMSwiZiIsMCx7InN0eWxlIjp7ImJvZHkiOnsibmFtZSI6ImRhc2hlZCJ9fX1dLFswLDJdLFsxLDJdXQ==
\begin{equation}\label{eq_abelian_fibered_CY_example}
\begin{tikzcd}
	X && X \\
	& {B}
	\arrow["f", dashed, from=1-1, to=1-3]
	\arrow[from=1-1, to=2-2]
	\arrow[from=1-3, to=2-2]
\end{tikzcd}    
\end{equation}
which act on the general abelian fiber as multiplication by an appropriately chosen integer. In the case where \(X\) is a K3 surface and \(\pi\) an elliptic fibration on \(\proj^1\) this is a classical construction, see for example \cite[Section 3]{BogTsc1999}. Such rational endomorphisms of \(X\), over the smooth locus \(U\subset \proj^1\), is the quotient by a group scheme \(G\)  finite over \(U\). However this does not necessarily make \(f\) Galois as the group scheme \(G\) might not be trivial (see Section \ref{sec_abelian_fibrations}).

\end{eg} 

\subsection{Main results}
Our first result answers Question \ref{quest_Galoiscover} in the negative, by providing examples of \hk{} manifolds which admit no nontrivial CY Galois covers.
\begin{thm}\label{thm_HK2_rho=1}
    Let \(X\) be a \hk{} manifold with \(b_2(X)=23\) and \(\rho_X=1\). Let \(\phi \colon Y \dashrightarrow X\) be a Calabi-Yau rational cover. Then \(Y\) admits no nontrivial birational automorphism commuting with \(\phi\), i.e.,
    \[\Bir(\phi)=\{\id\}.\]
    In particular, if \(\phi\) is Galois, then it is an isomorphism.
\end{thm}
For a very general smooth cubic fourfold \(W\subset \proj^5\), the variety \(F=F_1(W)\) in Example~\ref{eg_voisinmap_F_1} satisfies all hypothesis of \(X\) in Theorem \ref{thm_HK2_rho=1}, from which we deduce the following.
\begin{cor}\label{cor_F_1generic}
    For a very general cubic fourfold \(W\subset \proj^5\), \(F_1(W)\) admits no nontrivial Calabi-Yau Galois covers. 
\end{cor}
Theorem \ref{thm_HK2_rho=1} also answers negatively Question \ref{quest_Galois_factors}.
\begin{cor}
For a very general cubic fourfold \(W\subset \proj^5\), the Voisin map \(\nu \colon F_1(W) \dashrightarrow F_1(W)\) is not a Calabi-Yau Galois-like cover.
\end{cor}

The second main result is devoted to the specific case of abelian Galois covers; in this setting Question \ref{quest_birat_Galois_factors} is equivalent to Question~\ref{quest_Galois_factors} (see Lemma \ref{lem_question_equivalence_for_abelian_covers}).
While the previous corollaries deal only with very general cubic fourfolds, the following applies to all smooth cubic fourfolds, even those whose Fano variety of lines admits an abelian Galois cover (we recall their construction in Section \ref{sec_abelian_bounds}). 
%Indeed, as we sketch in Section \ref{sec_abelian_bounds}, the well-developed theory of special cubic fourfolds and their periods (\cite{Voi1986}, \cite{Has2000}, \cite{Laz2009}) implies that for certain choices \(W\subset \proj^5\) there exists an abelian surface \(T\) and a Galois cover \[h\colon T^2\dashrightarrow F_1(W)\]which answers positively Question \ref{quest_Galoiscover}. 

\begin{thm}\label{thm_vosinmap_abeliancover}
    Let \(W\subset \proj^5\) be a smooth cubic fourfold. Then the Voisin map \(\psi \colon F_1(W) \dashrightarrow F_1(W) \) is not an abelian Galois-like cover.
\end{thm}
Above, as in Definition \ref{def_Galois_like}, we say that \(\psi: X'\dashrightarrow X\) is abelian Galois-like if there exist an abelian variety \(A\) and a rational cover \( \varphi \colon A \dashrightarrow X'\) such that \(\psi\circ \varphi\) is Galois.

Theorem \ref{thm_vosinmap_abeliancover} allows us to answer negatively Question \ref{quest_birat_Galois_factors}.

%In Section \ref{sec_abelian_bounds} we recall the existence of special cubic fourfolds whose Fano variety of lines admits an abelian cover

%the well-developed theory of special cubic fourfolds and their periods (\cite{Voi1986}, \cite{Has2000}, \cite{Laz2009}) implies that for certain choices \(W\subset \proj^5\) there exists an abelian surface \(T\) and a Galois cover \[h\colon T^2\dashrightarrow F_1(W)\]which answers positively Question \ref{quest_Galoiscover}. 

\begin{cor}
Let \(W\) be a cubic fourfold whose Fano variety of lines admits an abelian Galois cover
\(h \colon A\dashrightarrow F_1(W)\), and let \(\nu \colon F_1(W)\dashrightarrow F_1(W)\) be the Voisin map.
The composition \(\nu\circ h\) does not factor any Galois cover of Kodaira dimension zero.
\end{cor}

A different negative answer to Question \ref{quest_birat_Galois_factors} is
given at the end of Section \ref{sec_abelian_fibrations}, and concerns the rational maps in Example  \ref{eg_abelian_fibrations}, for K3 surfaces.

We now turn to the study of properties of the \textit{branch divisor} of rational covers, which we now define. 
Let \(\phi \colon X'\dashrightarrow X\) be a dominant rational map between CY manifolds, and fix a resolution of indeterminacies, meaning a diagram
% https://q.uiver.app/#q=WzAsMyxbMCwxLCJYJyJdLFsyLDEsIlgiXSxbMSwwLCJaIl0sWzAsMSwiZiIsMCx7InN0eWxlIjp7ImJvZHkiOnsibmFtZSI6ImRhc2hlZCJ9fX1dLFsyLDAsIlxcdGF1IiwyXSxbMiwxLCJcXHdpZGV0aWxkZXtmfSJdXQ==
\[\begin{tikzcd}
	& Z \\
	{X'} && X
	\arrow["\tau"', from=1-2, to=2-1]
	\arrow["{\widetilde{f}}", from=1-2, to=2-3]
	\arrow["f", dashed, from=2-1, to=2-3]
\end{tikzcd}\]
with \(Z\) smooth and \(\tau\) birational.
We denote by \(\Ram(\simf \,)\subset Z\) the ramification divisor. We define the branch divisor \(\Bran(f)\subset X\) as \(\simf_* \Ram(\simf \,)\). It does not depend on the choice of \(Z\) and \(\simf\). We say that a component \(P\subset |\Ram(\simf \,)|\) is \textit{essential} if \(\simf_*P\neq 0\), i.e., if \(P\) dominates a divisorial component of the branch locus of \(\simf\).  
Note that \(\Bran(f)\) is uniruled, since it is dominated by \(\Ram(\simf \,)\) and by the canonical bundle formula we have
\begin{equation}\label{eq_Ram_=_Exc}
    \Ram(\simf )=\Exc(\tau),
\end{equation}

where \(\Exc(\tau)\) denotes the exceptional divisor of \(\tau\).

%Note that the very general \hk{} fourfold in \({}^2\mathscr{M}^{(2)}_6\) is isomorphic to \(F_1(W)\) as above, for some \(W\subset \proj^5\).
% In words: the existence of Calabi-Yau Galois \textit{covers} of \(X\) does not imply the existence of Calabi-Yau Galois \textit{closures} of \(X'\dashrightarrow X\).
The next result is a criterion for a CY rational cover \(f\) to be CY Galois-like, based on the study of essential exceptional divisors.

\begin{prop}\label{prop_uniruled_above_branch}
In the above notation, suppose \(f\) is CY Galois-like. Then each essential irreducible component \(D \subset \widetilde{f}^{-1}(\Bran(\widetilde{f}))\) with its reduced structure is a uniruled variety.
\end{prop}

We apply this criterion to Example \ref{eg_voisinmap_LLSvS}, proving that the second Voisin map is not a Calabi-Yau Galois-like cover (see Corollary \ref{cor_uniruled_LLSvS}).

In our last result, we prove that the branch divisors of certain Galois covers of \hk{} manifolds are negative with respect to the Beauville-Bogomolov-Fujiki form:
\begin{thm}\label{prop_branch_is_exceptional}
Let \(f \colon Y\dashrightarrow X\) be a Calabi-Yau Galois cover of a \hk{} manifold, and suppose that \(\Gal(f)\subset \Aut(Y)\). Then \(\Bran(f)\) is \(q_X\)-exceptional (see Definition \ref{def_exceptional}).
\end{thm}
This theorem uses crucially the theory of Zariski decomposition on \hk{} manifolds developed in \cite{Bou2004}, as well as an important result of \cite{Dru2011}.
%Note that it is a priori hard to produce \(q_X\)-exceptional divisors, without knowing the explicit geometry of the \hk{} one is working with: another way is to exhibit 
%The advantage of our approach is that it can produce divisors that are singular, even non-normal, which are hard to study by geometric methods like foliations. 
Note that, unlike the smooth uniruled divisors appearing in the work of Markman (see for example \cite[page 11]{AneHuy2024}), the \(q_X\)-exceptional divisors obtained this way are typically singular. This is the case of the Chow divisors in both K3\(^{[n]}\) and Kum\(_n\). 
%They are even non-normal for \(n\geq 3\) (????).

The assumption on \(\Gal(f)\) in Theorem  \ref{prop_branch_is_exceptional} is always satisfied for \(Y\) abelian. As a consequence, we improve the known lower bounds on the second Betti number of any \hk{} manifold satisfying Laza's condition (see Question \ref{quest_laza}):
\begin{cor}
    If a \hk{} manifold \(X\) has a deformation admitting an abelian Galois cover, then
    \[b_2(X)\geq 4.\]
    More precisely, if \(\phi \colon A\dashrightarrow X'\) is an abelian Galois cover for some deformation \(X'\) of \(X\) and if we denote by \(c\geq 1\) the number of distinct irreducible components of \(\Bran(\phi)\), then 
    \[b_2(X)\geq 3+c.\]
\end{cor}

\subsection{Overview}

In Section \ref{sec_HK_cover_bounds} we relate the cohomologies of a \hk{} manifold and its rational Calabi-Yau covers via Hodge theory.

In Section \ref{sec_CY_type} we use the cohomological restrictions from the previous section and the Beauville-Bogomolov decomposition theorem to reduce Question \ref{quest_Galoiscover} to those cases where \(Y\) is either \hk{} or abelian. We then specialize to the dimension \(4\) case, where we prove Theorem \ref{thm_HK2_rho=1}: a crucial tool we employ is the main result of \cite{JiaLiu2024}, which leverages strong topological restrictions for \hk{} manifolds (\cite{Gua2001}) and their singular counterparts (\cite{FuMen2021}).

In Section \ref{sec_abelian_bounds} we prove Theorem \ref{thm_vosinmap_abeliancover}. The main ingredient is the computation of the monodromy group of the Voisin map, made in \cite{GioGio2024}.

In Section \ref{sec_abelian_fibrations} we recall the construction mentioned in Example \ref{eg_abelian_fibrations}. Inspired by \cite[Appendix A]{KLM2023}, this provides counterexamples for Question \ref{quest_birat_Galois_factors}.

In Section \ref{sec_uniruled_abovebranch} we prove Proposition \ref{prop_uniruled_above_branch} and Corollary \ref{cor_uniruled_LLSvS}.  

In Section \ref{sec_exceptional_branch}, we prove Theorem \ref{prop_branch_is_exceptional}, in slightly greater generality than how it is stated above. 

\subsection{Conventions}
Throughout the text, algebraic variety means an integral scheme of finite type and separated over \(\C\).

A projective \textit{strict Calabi-Yau} (or \textit{CY}) \textit{manifold} is a projective Calabi-Yau manifold of dimension \(\geq3\) with \[h^{p,0}(X)=0 \ \ \textrm{ for all } p=1,\dots \dim X -1 .\]

A \textit{projective \hk} manifold is a projective simply connected complex manifold \(X\) which admits a nowhere vanishing holomorphic 2-form \(\sigma\) such that\[\C \, \sigma= H^{2,0}(X).\]

\subsection*{Acknowledgements}
I would first like to thank Claire Voisin for many invaluable conversations and for sharing original ideas that motivated and shaped this work. I am also grateful to Emanuele Macrì for his mentorship and his precious advice. I would like to thank Radu Laza for his fascinating talk in Cetraro, September 2024, during which he posed some of the questions discussed here. I am particularly indebted towards Franco Giovenzana for sharing his experience and knowledge, as well as Francesco Denisi and Ángel David Ríos Ortiz for many enriching conversations. I finally thank Lisa Marquand and Moritz Hartlieb for spotting mistakes in earlier versions. The author benefited from the support of ERC Synergy Grant 854361 HyperK.
 
\section{Proof of Theorem \ref{thm_HK2_rho=1}}
We prove in this section Theorem \ref{thm_HK2_rho=1}. The proof will be divided in two parts.
\subsection{Topological bounds on CY rational covers of HK manifolds}\label{sec_HK_cover_bounds}

We recall an important object for studying the cohomology of smooth projective varieties related by rational maps: for any \(\Z\)-Hodge structure \(H\) of weight 2 and \(H^{p,q}=0\) for \(|p-q|>2\), the \textit{transcendental lattice} \[T(H)\subset H\] is defined as the smallest primitive \(\Z\)-Hodge substructure such that \[H^{2,0}\subset T(H)\otimes \C.\] 
\begin{df}
Let \(X\) a smooth projective manifold. We define its transcendental lattice as
\[T(X):=T(H^2(X,\Z)).\]
\end{df}
This sub-Hodge structure enjoys nice properties when it comes from K3-type polarized Hodge structures.
\begin{prop}\label{prop_T_properties}
    Let \(X\) be a projective \hk{} manifold.
    Then the Hodge structure \(T(X)\) is polarized and irreducible, where the polarization is induced by the Beauville-Bogomolov form \(q_X\). Moreover,
    \[T(X)=\NS(X)^{\perp_{q_X}}.\]
\end{prop}
\begin{proof}
    The first claim follows from the definition of \hk{} manifold and \cite[Lemma 2.7]{Huy2016}. The second claim is proven in the \(\dim X =2\) case in \cite[Lemma 3.1]{Huy2016}. However, the same proof goes through just as well in any dimension, if one replaces the intersection form with the form \(q_X\).
\end{proof}
We denote the rational cohomology classes spanned by the transcendental lattice as \[T(X)_\Q := T(X)\otimes \Q \subset H^2(X,\Q).\]
This is a \(\Q\)-irreducible Hodge structure.

For any \(\Q\)-vector subspace \(V\) of some graded \(\Q\)-algebra \(R^*\), we denote by \(S^*V\) the subalgebra generated by \(V\).
The following statement is classical: we provide a proof for completeness.

Let \(X,Y\) be smooth projective complex varieties of dimension \(2n\) and \(\phi \colon Y\dashrightarrow X\) be a dominant rational map. Then one defines a morphism of \(\Q\)-Hodge structures 
\[\phi^* \colon H^k(X,\Q) \rightarrow H^k(Y,\Q)\]
as follows.
Fix a resolution of indeterminacies 
% https://q.uiver.app/#q=WzAsMyxbMCwxLCJZIl0sWzIsMSwiWCJdLFsxLDAsIlxcd2lkZXRpbGRle1l9Il0sWzAsMSwiXFx2YXJwaGkiLDAseyJzdHlsZSI6eyJib2R5Ijp7Im5hbWUiOiJkYXNoZWQifX19XSxbMiwwLCJcXHRhdSIsMl0sWzIsMSwiXFx3aWRldGlsZGV7XFx2YXJwaGl9Il1d
\begin{equation}\label{eq_resolution}
\begin{tikzcd}
	& {\widetilde{Y}} \\
	Y && X
	\arrow["\tau"', from=1-2, to=2-1]
	\arrow["{\widetilde{\phi}}", from=1-2, to=2-3]
	\arrow["\phi", dashed, from=2-1, to=2-3]
\end{tikzcd}    
\end{equation}
Then we set
\[\phi^*:=\tau_*\widetilde{\phi}^* \]
where \(\tau_*\) is the Gysin morphism. 

\begin{prop}\label{prop_inj_subalgebras}

Let \(X,Y\) be smooth projective complex varieties of dimension \(2n\) and \(\phi \colon Y\dashrightarrow X\) be a dominant rational map.
     Then the morphism \(\phi^*\) induces by restriction a morphism
    \begin{equation}\label{eq_inj_subalgebras}
    S^*T(X)_\Q \xrightarrow{\phi^*} H^*(Y,\Q)
\end{equation}
which is injective and compatible with cup-product.

\end{prop}

\begin{proof}
We use the notation from diagram (\ref{eq_resolution}).
The proposition will be a consequence of the following:
\begin{lem}\label{lem_T(X)_inclusion}
    For any \(\alpha \in T(X)_\Q\), we have
    \begin{equation}\label{eq_alpha_identity}
        \simvarphi^* \alpha = \tau^* \varphi^* \alpha \ \ \ \textup{in }H^2(\simY,\Q).
    \end{equation}
\end{lem}
\begin{proof}
    By the blow-up formula, we have a decomposition as Hodge structures
    \[H^2(\simY,\Q)=\tau^* H^2(Y,\Q)\oplus N\]
    where \(N\) consists only of divisor classes.
    Let \(\pi_N\) be the projection on the second term.
    The composition 
    \[\pi_N \circ \simvarphi^* \colon T(X)_\Q \rightarrow N\]
    sends \(H^{2,0}(X)\subset T(X)\otimes \C\) to \(0\) since the Hodge structure on \(N\) is trivial. Being a Hodge substructure which contains \(H^{2,0}(X)\), the kernel of \(\pi_N \circ \simvarphi^*\) must then coincide with \(T(X)_\Q\) by minimality of the Hodge structure \(T(X)_\Q\). Hence \(\pi_N\circ \simvarphi^*=0\).
    This implies that \[\simvarphi^* T(X)_\Q \subset \tau^*H^2(Y,\Q) \ \ \ \textup{in }H^2(\simY , \Q).\]

    We conclude by the projection formula \(\tau_* \circ \tau^*= \id_{H^2(Y,\Q)}\).
\end{proof}
Since both \(\simvarphi^* \colon H^*(X,\Q)\rightarrow H^*(\widetilde{Y},\Q)\) and \(\tau^* \colon H^*(Y,\Q)\rightarrow H^*(\widetilde{Y},\Q)\)
are compatible with cup products, formula (\ref{eq_alpha_identity}) provides for any \(\alpha \in S^*T(X)_\Q\)
\begin{equation}\label{eq_truc2}
        \simvarphi^* \alpha = \tau^* \varphi^* \alpha^* \ \ \ \textup{in }H^{2*}(Y,\Q).
    \end{equation}
    
As \(\tau^*\) is injective, (\ref{eq_truc2}) implies that \(\varphi^* \colon S^*T(X)_\Q \rightarrow H^*(Y,\Q)\) respects cup products.
As \(\simvarphi^* \colon H^*(X,\Q)\rightarrow H^*(\widetilde{Y},\Q)\) is injective, because \(\simvarphi\) is a surjective morphism, it follows that \(\varphi^* \colon S^*T(X)_\Q \rightarrow H^*(Y,\Q)\) is also injective.

\end{proof}
\begin{cor}\label{cor_betti_bounds}
With the same hypothesis as Proposition \ref{prop_inj_subalgebras}, for any \(k\leq n\) we have
\[b_{2k}(Y)\geq \binom {b_2(X)-\rho_X+k-1}{k} , \]
\end{cor}
\begin{proof}
By work of Verbitsky (see \cite{Bog1996}), we know that 
\begin{equation}\label{eq_verbitsky_relations}
    S^kT(X)_\Q\simeq\Sym^k T(X)_\Q
\end{equation}
The conclusion follows from Proposition \ref{prop_inj_subalgebras}.

\end{proof}

\subsection{CY covers of HK fourfolds}\label{sec_CY_type}

Let \(Y\dashrightarrow X\) be a CY rational cover. 
By the Beauville-Bogomolov decomposition theorem, there exists a finite étale cover 
\begin{equation}\label{eq_BB_decomp}
    A\times C\times H \rightarrow Y
\end{equation}
with \(A\) abelian, \(C\) a product of strict Calabi-Yau manifolds and \(H\) a product of \hk{} manifolds. 
Note that the Galois closure of the map (\ref{eq_BB_decomp}) is obtained by taking a finite étale cover of the abelian part, which yields another cover of the form \(A'\times C\times H\).

It follows that, in order to classify all CY Galois covers \(Y\dashrightarrow X\), we need only to classify Galois covers of the form \(A\times C\times H\dashrightarrow X\) and the normal subgroups \[H\lhd \Aut(A\times C\times H / X)\subset \Gal(A\times C\times  H/X)\]
that act freely on \(A\times C\times H\).

We assume now that \(X\) carries a holomorphic symplectic 2-form.
It then follows that \(H^{2,0}(A\times C\times H)\) contains a nowhere degenerate holomorphic 2-form, which implies that \(C\) is reduced to a point, since \(H^{2,0}(C)=H^{1,0}(C)=0\). 
Let us denote by \(e\) the dimension of \(A\) and let us write
\[H=H_1\times \cdots \times H_r\]
with each \(H_i\) a \hk{} manifold of dimension \(2k_i\). 

At the extreme case where \((e,r)=(0,1)\), we have that \(Y\) admits a finite étale cover from a \hk{} manifold: by multiplicativity of the holomorphic Euler characteristic and the fact \(h^{2,0}(Y)\neq 0\), the degree of \(H\rightarrow Y\) is one (see for example \cite[Proposition 2.6]{OguSch2011}) hence 
\[(e,r)=(0,1) \Rightarrow Y \textup{ is \hk}.\]

%\begin{rmk}\label{rmk_HE}
%To the opposite end of the spectrum we have the case \((e,r)=(\dim X,0)\): the study of these varieties in general dimension was started by \cite{Lan2001}, which defines a complex manifold \(Y\) to be \textit{hyperelliptic} when it admits a finite étale cover \(A\rightarrow Y\) with \(A\) an abelian variety.
%The name originates from the case \(\dim Y =2\) (as defined, for example, at pag. 585 of \cite{GriHar1978}). These surfaces where completely classified in the classical work of \cite{BagdeF1908}. This task boils down to classifying free, translation-free actions of finite groups by biholomorphisms on abelian varieties, and has been carried out in dimension 3 (\cite{UchYos1976},\cite{CatDem2020}) and partly in dimension 4 (\cite{Dem2022}). 
%Beware that hyperelliptic curves are never \virgolette hyperelliptic one dimensional varieties" in the present sense, by the Riemann-Hurewicz formula.
%\end{rmk}

\begin{df}\label{def_CY_type}
Let \(X\) be a Calabi-Yau manifold with \(h^{2,0}(X)\neq 0\).
We say that \(Y\dashrightarrow X\) is a CY rational cover \textit{of type }\((e,r)\) if there exists a finite étale morphism
\begin{equation}\label{eq_BB_over_HK}
    A\times H_1\times \cdots \times H_r \rightarrow Y
\end{equation}
for some abelian variety \(A\) of dimension \(e\) and some \hk{} manifolds \(H_i\).
%If this map is only rational we add the prefix `` \textit{rational}" to `` \textit{cover}"
\end{df}
\noindent By the uniqueness of Beauville-Bogomolov decomposition, this is well defined. 

%\subsection{CY covers of HK fourfolds}\label{subß}
Let us now specialize the above discussion to the case of \hk{} fourfolds with high second Betti number.

\begin{prop}\label{prop_type_over_HK4fold}
    Let \(Y\dashrightarrow X\) be a CY rational cover of type \((e,r)\) with \(X\) a \hk{} manifold of dimension \(4\) and \(b_2(X)=23\). 
    \begin{enumerate}
        \item If \(\rho_X\leq 11\), then  \[(e,r)\in \{(2,1),(0,2),(0,1)\}\]. 
        \item If \(\rho_X \leq 7 \), then \(e=0\).
        \item \label{item_second_last} If \(\rho_X=1\), then \((e,r)=(0,1)\) and \(Y\) is \hk.
        \item \label{item_last} If \(r=2\) then \(Y\) is isomorphic to the product of two K3 surfaces.
    \end{enumerate}

    %Moreover, all the assumptions on \(\rho_X\) are optimal.
\end{prop}
\begin{proof}
Items \textit{(1),(2)} and \textit{(3)} are immediate consequences of Corollary \ref{cor_betti_bounds} and the remark just before Definition \ref{def_CY_type}. For the claim \textit{(\ref{item_last})}, let 
\(H_1\times H_2 \xrightarrow{ \ f \ } Y\) be a finite étale map with both \(H_i \) being K3 surfaces.
We have \(h^{3,0}(Y)=0\) since \(b_3(H_1\times H_2)=0\), hence 
\[\chi(\Ocal_Y)=2+h^{2,0}(Y)\geq 3.\]
But \(\deg (f) \chi(\Ocal_Y)=\chi(H_1\times H_2)=4 \), so \(f\) is an isomorphism.
\end{proof}

Item \textit{(\ref{item_second_last})} is optimal in the sense that there exist CY rational covers of type \((0,2)\) if \(\rho_X>1\): indeed, for any K3 surface, \(S\) one has a rational (even Galois) map
\[S\times S \dashrightarrow S^{[2]},\]
and as soon as \(\rho_S=1\) one has \(\rho_{S^{[2]}}=2\). 
It would be interesting to establish if the other bullet points are optimal too.

We can now prove our first main result:
\begin{proof}[Proof of Theorem \ref{thm_HK2_rho=1}]

By Lemma \ref{prop_type_over_HK4fold}, \(Y\) is a \hk{} manifold of dimension 4. By Corollary \ref{cor_betti_bounds} for \(k=1\) and \cite{Gua2001}, \(Y\) has the same Betti numbers as \(X\), i.e., 
\[b_2(Y)=23, \ b_3(Y)=0, \ b_4(Y)=278.\]

The injective morphism of Hodge structure \(\phi^* \colon T(X)_\Q \rightarrow H^2(Y,\Q)\) is an isomorphism onto \(T(Y)_\Q\), as both Hodge structures are irreducible and \(\phi^* (T(X)_\Q)\) contains a nonzero class, namely \(\phi^*\simga_X=\sigma_Y\). Since \(\rho_X=1\) and \(b_2(Y)=23=b_2(X)\), this implies \(\rho_Y=1\), so \(\NS(Y)\) is spanned by an ample class we denote by \(\omega_Y\). If \(g\in \Bir(\phi)\), then the action of \(\phi^*\) on \(H^2(Y,\Q)\) fixes \(T(Y)_\Q\) and stabilizes \(\NS(Y)_\Q\), so \(\phi^* \omega =\lambda \omega\) for some nonzero integer \(\lambda\). 
We have \(\lambda \in \Z^\times \) since \(g^*\) is an automorphism of \(H^2(Y,\Z)\), and \(\lambda > 0\) since \(g^*\) sends effective divisors to effective divisors. We conclude that \(\lambda =1\), so \(g^* \) acts trivially on the whole \(H^2(Y,\Z)\). 
Once again by Corollary \ref{cor_betti_bounds},
inspecting the Betti numbers of \(Y\) one concludes that the entire cohomology ring of \(Y\) is generated by degree 2 cohomology, so \(g^*\) acts trivially on the whole cohomology ring \(H^*(Y,\Q)\). The conclusion then follows by \cite[Theorem 1.1]{JiaLiu2024}, which says that such a \(g\) must be the identity.
\end{proof}

%We observe that under the stronger assumption of \(X\) being of K3 deformation type, the proof of Theorem \ref{thm_HK2_rho=1} can also be completed by appealing to more well-known results than the recent \cite{JiaLiu2024}, namely the fact that for any \hk{} manifold \(X\), one has that the birational automorhpisms on \(X\) acting trivially on \(H^2(X,\Z)\) are everywhere defined and that the group of such automorphisms is deformation invariant (\cite[Theorem 2.1]{HasTsc2013}), together with the fact that for \(X=\textrm{K3}^{[n]}\), any automorphism that acts trivially on \(H^2\) is trivial (\cite[Proposition 10]{Bea1982}). 

\section{Bounds on automorphisms of abelian Galois covers}\label{sec_abelian_bounds}

In this section we focus on abelian rational covers. Here is an important remark, which motivates the study of this particular case:
\begin{lem}\label{lem_question_equivalence_for_abelian_covers}
Let \(f \colon Y \dashrightarrow A\) be a rational cover, where \(Y\) is a smooth projective variety of Kodaira dimension zero and \(A\) is abelian. Then \(Y\) is birational to an abelian variety and \(f\) is Galois. 
\end{lem} 
\begin{proof}
Any dominant rational map \(f \colon Y\dashrightarrow A\) is in fact everywhere defined since \(A\) contains no rational curve. By the universal property of the Albanese morphism \(\alb_Y \colon Y \rightarrow Alb(Y)\), \(f\) must factor as
% https://q.uiver.app/#q=WzAsMyxbMCwwLCJZIl0sWzIsMCwiXFxBbGIoWSkiXSxbMSwxLCJBIl0sWzAsMSwiXFxhbGJfWSJdLFsxLDIsImciXSxbMCwyLCJmIiwyXV0=
\[\begin{tikzcd}
	Y && {\Alb(Y)} \\
	& A
	\arrow["{\alb_Y}", from=1-1, to=1-3]
	\arrow["f"', from=1-1, to=2-2]
	\arrow["g", from=1-3, to=2-2]
\end{tikzcd}.\]
Clearly \(g\) is dominant as \(f\) is, and by \cite{Kaw1981} we know that \(\alb_Y\) is a fibration, thus \[\dim A = \dim Y \geq \dim \Alb(Y)\geq \dim A .\]
We deduce that \(g\) is an isogeny and \(\alb_Y\) is birational, proving the first statement.
Any dominant morphism between abelian varieties of the same dimension is an isogeny, hence is Galois, which proves the second statement. 
\end{proof}

\begin{cor}
    In the notation of Questions \ref{quest_Galois_factors} and \ref{quest_birat_Galois_factors}, the two questions are equivalent whenever \(X'\) is abelian.
\end{cor}

%\color{blue}

%For any positive integer \(d\), we write it as \(d=r_1\cdots r_k\) so that \(r_j\) are pairwise coprime prime powers, and we set
%% \alpha(d):= \sum_{i=1}^k \varphi(r_i)
%\end{equation}
%where \(\varphi\) denotes the Euler totient function.

For a finite order automorphism \(M\) of a finite dimensional \(\Q\)-vector space, one can bound \(\ord(M)\) indirectly by the following simple observation involving the value of the Euler totient function \(\varphi(\ord(M))\). The bound one gets this way is optimal. We will only need this in the case where \(\ord(M)\) is a prime power, but one can formulate a more general, albeit less clean statement.
\begin{lem}\label{lem_GLZ_period}
Let \(M\in GL_m(\Q)\) be of finite order equal to \(p^a\), where \(p\) is a prime and \(a\in \N\). Then \begin{equation}\label{eq_truc}
    \varphi(p^a)=(p-1)p^{a-1} \leq m.\end{equation}
\end{lem}

\begin{proof}
%Let \(d=r_1\cdots r_k\) be the factorization into pairwise coprime prime powers. 
Every eigenvalue \(\xi\) of \(M\) over \(\C\) is a \(p^a\)-th root of unity, and by the minimality of \(d\), there must be one \(\xi\) which is a primitive \(p^a\)-th root of unity. 
%\[d=\textup{lcm} \{\ord (\xi): \xi \textup{ eigenvalue of } M \}.\]

It follows that the cyclotomic polynomial \(\Phi_{p^a}\), which is irreducible over \(\Q\) with degree equal to \(\varphi(p^a)=(p-1)p^{a-1}\), divides the minimal polynomial \(q\in \Q[X]\) of \(M\). As the degree of \(q\) is at most \(m\), we get (\ref{eq_truc}).
\end{proof} 

\begin{cor}\label{cor_bound_abelian_automorphism}
    Let \(A\) be an abelian variety and \(f\in \Aut_0(A)\) an automorphism fixing zero of finite order equal to \(p^a\), where \(p\) is prime . Then
    \[(p-1)p^{a-1}\leq 2 \dim A.\]
\end{cor}
\begin{proof}
%One can write \(f=\tau_v \circ F\), where \(\tau_v\) is the translation by \(v:=f(0)\in A\) and \(F\in \Aut_0(A)\) is an automorphism fixing \(0\in A\). Let \(e=\ord(F)\) and \(d=\ord(f)\). Then \(e|d,\) and for any \(x\in A\)\[f^e(x)=x+\sum_{i=0}^{e-1}F^i(v).\] It follows that \(\sum_{i=0}^{e-1}F^i(v)\) is of \(\frac{d}{e}\)-torsion.hence \(e=d\). For this reason we may assume \(f=F\), i.e., \(f\in \Aut_0 (A)\).

From the basic theory of abelian varieties, we now that the group homomorphism
\[\Aut_0(A) \rightarrow \GL(H^1(A,\Z))\]
\[f \mapsto f^*\]
is injective, hence \(\ord(f)=\ord(f^*)\). Since \(\rk H^1(A,\Z)=2\dim A\), we conclude by Lemma \ref{lem_GLZ_period}.
\end{proof}

\begin{rmk}
    The above Corollary is probably well-known. The only similar statement the author could locate in the literature is \cite[Exercise 4, Chapter 5]{BirLan2004}, which asserts that for any finite order automorphism \(f\in \Aut_0(A)\), without restrictions on prime factors imposed in Corollary \ref{cor_bound_abelian_automorphism}, one has the stronger bound
    \[\varphi(\ord(f))\leq 2 \dim A.\]
    However, this is false\footnote{many thanks to Francesca Rizzo and Olivier Benoist for a conversation that led to this remark.}: a counterexample can be constructed as follows. Choose an elliptic curve \(E\) with automorphism \(f\) of order \(3\), and choose an abelian surface \(T\) with automorphism \(g\) of order \(5\). These abelian varieties of dimension \(\frac{\varphi(n)}{2}\) can be constructed starting from the ring of integers of cyclotomic number fields \(\Q(\xi_n)\), as in \cite[Proposition 13.3.1]{BirLan2004}.

    We consider the abelian threefold \(A:=E\times T\) with the automorphism \((f,g)\) of order \(15\). Then 
    \[\varphi(15)=8>6=2\dim(A),\]
    contradicting the above claim.
\end{rmk}

Before proving Theorem \ref{thm_vosinmap_abeliancover}, we give some background on how to construct an abelian Galois cover \(\phi \colon A\dashrightarrow F_1(W)\) for some smooth cubic fourfolds \(W\). According to \cite[Thm 2]{Add2016}, we know that by picking \(W\) onto those Hassett divisors \(\Ccal_d\) of the moduli space of cubics for which \(d\) fulfills a certain numerical condition, we can have \(F_1(W)\) be birational to \(S^{[2]}\), for some \(S\). The K3 surface \(S\) is determined by the fact its primitive middle cohmology lattice is identified with a sublattice of corank 1 inside the primitive middle cohomology of \(W\).
For this reason, by the surjectivity (outside some Heegner divisor) of the period map of polarized K3 surfaces, we can choose \(W\) in \(\Ccal_d\) so that \(H^2(S,\Z)\) contains a Kummer lattice. It follows that \(S\) is the Kummer surface associated to some abelian surface \(T\), and we get that \(F_1(W)\) is birational to \(\Km(T)^{[2]}\). One can then take for \(\phi\) the composition of this birational map with the natural Galois map
\[T^2 \dashrightarrow \Km(T)^{[2]}\]

We now give the proof of Theorem \ref{thm_vosinmap_abeliancover}, of which we repeat the statement here for the reader's convenience.
\begin{thm*}[Theorem \ref{thm_vosinmap_abeliancover}]
    Let \(W\subset \proj^5\) be a smooth cubic fourfold. Then the Voisin map \(\nu \colon F_1(W) \dashrightarrow F_1(W) \) is not an abelian Galois-like cover.
    \end{thm*}
\begin{proof}
Let us denote \(F_1(W)\) as \(F_1\). Suppose by contradiction that there exists a rational cover \(\psi \colon A\dashrightarrow F_1\) such that \(\nu \circ \psi\) is Galois. 
If we denote by \(\K\) the Galois closure of \(\nu^* \colon \C(F_1) \hookrightarrow \C(F_1)\) inside \((\nu\circ \psi)^* \colon \C(F_1) \hookrightarrow \C(A)\), we have the following surjective homomorphism of groups
\[\Bir (\nu\circ \psi)=\Gal (\nu\circ \psi) \twoheadrightarrow \Gal(\K /\C(X))=\Gal(\nu).\]
where the first equality follow from the Galois hypothesis. 
%For any \(g\in \Gal(\nu)\) and any preimage \(f\in \Bir(\nu \circ \psi)\), one has \(\ord(g)|\ord(f)\) which means \(\alpha(\ord(g))\leq \alpha (\ord(f))\). But \(\Bir(\nu\circ \psi)\subset \Aut(A)\), so by Corollary \ref{cor_bound_abelian_automorphism}, we have
%\begin{equation}\label{eq_suppose_abelian_galois}
%    \alpha(\ord(g))\leq \alpha(\ord(f))\leq 2\dim A =8.\end{equation}
Importantly, by \cite{GioGio2024} we have
\[\Mon (\nu)=S_{16}.
\]
By the above surjectivity and the fact  \(\Gal(\nu)=\Mon(\nu)\) (\cite{Har79}), there should exist elements \(g,h\in \Bir(\nu\circ \psi)\subset \Aut(A)\) of order \(11\) and \(13\) respectively. We can consider \(g_0,h_0 \in \Aut_0(A)\) such that \[g(x)=g_0(x)+g(0) \ \ , \ \ h(x)=h_0(x)+h(0)\] for all \(x\in A\). We claim it is not possible that both \(h\) and \(g\) are translations. Indeed, if they were both translations they would commute, but seeing \(g,h\) as cycles in \(S_{16}\) of order \(11,13\) respectively, we see immediately they cannot commute, since commuting cycles either have disjoint supports, or the same support. 

If \(g\) is not a translation, then \(\ord(g_0)=11\), in particular
\[\varphi(\ord (g_0))=10>8= 2\dim (A).\]
But this is in contradiction with Corollary \ref{cor_bound_abelian_automorphism}. Similarly, if \(h\) is not a translation then we must have \(\ord(h_0)=13\) and we find again the contradiction
\[\varphi(\ord(h_0))=12>8= 2\dim (A).\]
This concludes the proof.
%[Proof of Theorem \ref{thm_vosinmap_abeliancover}]
%The crucial point is that , we have \[\Mon(\nu)=S_{16}.\] 
%Then by \cite{Har79}, \(\Gal (\nu)=S_{16}\), so if by contradiction there existed an abelian cover \(\psi\colon A \dashrightarrow  F_1(W)\) such that the composition \(\nu\circ \psi\) is Galois, then we could pick an automorphism \(g\in \Gal(\nu)\) of order \(11\), and apply Proposition \ref{prop_abelian_order_bound} to deduce \[10=\alpha(11)\leq 2\dim A =8\] which is a contradiction.
\end{proof}

\section{Some elliptic K3 examples}\label{sec_abelian_fibrations}
We recall some well-known constructions regarding elliptic fibrations (see for example \cite[Section 3]{BogTsc1999}).
Consider an elliptic K3 surface 
\[\pi \colon S \rightarrow \proj^1\] with a multisection \(D\subset S\) of degree \(d\). 
Let \(U\subset \proj^1\) denote the smooth locus of \(\pi\) and \(V=\pi^{-1}(U)\) its preimage.
For each integer \(m\), we denote 
\[\Jac^m(V)\rightarrow U\]
the relative Jacobian (of degree \(m\)) of the smooth projective family \(\pi \colon V \rightarrow U\), parametrizing line bundles of degree \(m\) on the fibers of \(\pi\).

For each integer \(k\), there is a rational endomorphism
\[f_k\colon \Jac^1(X) \dashrightarrow \Jac^1(X)\rule{26pt}{0pt}\]
\[ \rule{57pt}{0pt}X_s\ni x \longmapsto (kd+1)x - kD_{|X_s}\ .\] 

which up to the birational identifications coincides with the quotient by the action on \(V\) of the subgroup scheme of the Jacobian
\[G:=\Jac^0(V)[kd+1]=\Ker(\Jac^0(V)\xrightarrow{\cdot kd+1} \Jac^0(V))\]
%On each smooth fiber \(X_s\), this is just the quotient by the natural action of the finite torsion subgroup\[\Jac^0(X_s)[kd+1]= H^1 _{\textrm{ét}}(X_s,\mu_{kd+1}).\] If we denote by \(U\subset \proj^1\) the smooth locus of \(\pi\),  the étale local system of groups \[G:=R^1\pi_* \underline{\mu_{kd+1}}_{\pi^{-1}(U)}\] is actually a finite group scheme over \(U\), whose action on \(\pi \colon \pi^{-1}(U)\rightarrow U\), a nontrivial \(Jac^0(\pi^{-1}(U))\)-torsor over \(U\), extends the fiberwise action described above.
Note that \(G\) is finite over \(V\), meaning that on each smooth fiber \(X_s\), \(f_k\) is just the quotient by the natural action of the finite torsion subgroup \(\Jac^0(X_s)[kd+1]\).

Thus the rational map \(f_k\) is Galois ``fiberwise", and will be actually Galois if the group scheme \(G/U\) is trivial, i.e., \[G\simeq \Jac^0(X_s)[kd+1]\times U\] as group schemes over \(U\), for one fixed \(s\in U\). In some cases, this \(f\) is Galois-like:
consider the following construction, very similar to the one of Kollár-Larsen in (\ref{ex_kollar_larsen}).
Let \(E_1\) and \(E_2\) be two elliptic curves, set \[X=\Km(E_1\times E_2)\] 
with \(g \colon E_1\times E_2 \dashrightarrow X\) the natural quotient map. Let \(\pi \colon X\rightarrow \proj^1\) be the map induced by the natural projection \(p_1 \colon E_1\times E_2\rightarrow  E_1\) at the level of quotients by the \(\pm 1\) involution \(i\).
Multiplication by \(m>2\) on the \(E_2\) component induces a rational map \(f \colon X\dashrightarrow X\), which is a particular case of the construction above. We have the diagram

% https://q.uiver.app/#q=WzAsNSxbMSwwLCJYIl0sWzIsMCwiWCJdLFsxLDEsIlxccHJval4xIl0sWzAsMSwiRV8xIl0sWzAsMCwiRV8xXFx0aW1lcyBFXzIiXSxbMCwxLCJmIiwwLHsic3R5bGUiOnsiYm9keSI6eyJuYW1lIjoiZGFzaGVkIn19fV0sWzAsMiwiXFxwaSIsMl0sWzEsMiwiXFxwaSJdLFszLDJdLFs0LDAsImciLDAseyJzdHlsZSI6eyJib2R5Ijp7Im5hbWUiOiJkYXNoZWQifX19XSxbNCwzLCJwXzEiLDJdXQ==
\[\begin{tikzcd}
	{E_1\times E_2} & X & X \\
	{E_1} & {\proj^1}
	\arrow["g", dashed, from=1-1, to=1-2]
	\arrow["{p_1}"', from=1-1, to=2-1]
	\arrow["f", dashed, from=1-2, to=1-3]
	\arrow["\pi"', from=1-2, to=2-2]
	\arrow["\pi", from=1-3, to=2-2]
	\arrow[from=2-1, to=2-2]
\end{tikzcd}.\]
where \(f\circ g\) is Galois, with Galois group generated by \(i\) and translations by \(m\)-torsion points on \(E_1\). This proves that \(f\) is a Calabi-Yau Galois-like cover. However, as in the introduction, \(f\) is not Galois.

Let us now give an example where \(f\colon X\dashrightarrow X\) constructed as above is not CY Galois-like.

Let \(C\rightarrow \proj^1\) be the smooth compactification of the degree 3 cover given by projecting on the \(x\)-coordinate the curve cut by the equation
\[y^3=(x-a_1)(x-a_2)(x-a_3)(x-a_4)(x-b_1)^2\]
branched along the 5 points \(a_1,\dots, a_4,b_1 \in \proj^1\). This has genus \(3\) (see for example \cite[Section 4]{Moo2016}) and has an obvious automorphism \(\eta\) of order 3 commuting with the map to \(\proj^1\), given by multiplying \(y\) by a cube root of unity. Choose an elliptic curve \(E\) with an automorphism \(\alpha\) of order \(3\) and define \[\pi \colon X\rightarrow \proj^1\] to be the unique smooth minimal model for the fibration
\[\faktor{E\times C}{(\alpha, \eta)} \rightarrow \faktor{C}{\eta}\]
where \((\alpha,\eta)\) acts diagonally on the product and has order 3. Then \(X\) is a K3 surface and \(\pi\) is an isotrivial fibration with a section (\cite[Appendix A]{KLM2023}), whose monodromy is trivialized by base changing along the natural map

\[q \colon C\rightarrow \faktor{C}{\eta}\simeq \proj^1.\]

\begin{prop}
    \(X\) being as above, for any \(m>2\), the fiberwise multiplication by \(m\) map \(f_m \colon X \dashrightarrow X\) is not a CY Galois-like cover.
\end{prop}
\begin{rmk}
    It follows by Remark \ref{rmk_surfaces} that \(f_m\) does not factor any rational cover \(Y\dashrightarrow X\) with \(\kappa (Y)=0\).
\end{rmk}

\begin{proof}
 
We are going to prove that the Galois closure of \(f_m\) has Kodaira dimension at least \(1\). 
Denote by \(g\colon Z\rightarrow X\) the base change of \(q\) along \(\pi\). We end up with the following diagram:
% https://q.uiver.app/#q=WzAsNSxbMSwwLCJYIl0sWzIsMCwiWCJdLFsxLDEsIlxccHJval4xIl0sWzAsMSwiQyJdLFswLDAsIloiXSxbMCwxLCJmIiwwLHsic3R5bGUiOnsiYm9keSI6eyJuYW1lIjoiZGFzaGVkIn19fV0sWzAsMiwiXFxwaSIsMl0sWzEsMiwiXFxwaSJdLFszLDIsIjM6MSIsMl0sWzQsMCwiZyJdLFs0LDNdLFs0LDIsIiIsMix7InN0eWxlIjp7Im5hbWUiOiJjb3JuZXIifX1dXQ==
\[\begin{tikzcd}
	Z & X & X \\
	C & {\proj^1}
	\arrow["g", from=1-1, to=1-2]
	\arrow[from=1-1, to=2-1]
	\arrow["\lrcorner"{anchor=center, pos=0.125}, draw=none, from=1-1, to=2-2]
	\arrow["f_m", dashed, from=1-2, to=1-3]
	\arrow["\pi"', from=1-2, to=2-2]
	\arrow["\pi", from=1-3, to=2-2]
	\arrow["{3:1}"', from=2-1, to=2-2]
\end{tikzcd}.\]
Note that the fibration \(Z\rightarrow C\) is generically isomorphic to the natural projection \(E\times C\rightarrow C\), since it has a section and trivial monodromy. In particular \(g\) is Galois and
\[\Gal(g)=\langle (\alpha, \eta) \rangle\simeq \faktor{\Z}{3\Z}.\]
We consider the fiberwise multiplication by \(m\) map \( f\colon X\dashrightarrow X\) and observe that 
\(f\circ g\) is Galois, while for \(m\geq 3\) \(f\) is not (as \((\alpha,\eta)\) is not stabilized by conjugation with the translation on the \(E\) component by a point that is not fixed by \(\alpha\)). 
These two facts together with \(\deg g=3\) being prime imply that \(\C(Z)\) is a Galois closure of \(f_m^* \colon \C(X)\hookrightarrow \C(X)\).
We now observe that the Kodaira dimension of \(Z\) is \(1\), as \(Z\) is birational to the product \(C\times E\) and the Kodaira dimension is additive for products of varieties.
%This is because it fibers over a base of general type, and Itaka's conjecture is known in this case (\cite{Vie1983}). 
%if \(Z'\) is any geometric realization of the Galois closure (i.e. any variety whose function field is the Galois closure of \(f^* \colon \C(X)\hookrightarrow \C(X)\)), then \(Z'\) is birational to \(Z\)  since \(g\) has degree \(3\) and \(f\) is not Galois. 

If \(f\) is a Calabi-Yau Galois-like cover, then there must be a dominant rational map \(P\dashrightarrow Z\), for some Calabi-Yau surface \(P\). However we have \(\kappa(P)\geq \kappa(Z)=1\), which is a contradiction.
\end{proof}

\section{Uniruledness above the branch divisor}\label{sec_uniruled_abovebranch}
This section is devoted to the proof of Proposition \ref{prop_uniruled_above_branch}.
For the reader's convenience, we recall the notation and the statement:
Let \(\phi \colon X'\dashrightarrow X\) a dominant rational map between CY manifolds, and fix a resolution of indeterminacies, meaning a diagram
% https://q.uiver.app/#q=WzAsMyxbMCwxLCJYJyJdLFsyLDEsIlgiXSxbMSwwLCJaIl0sWzAsMSwiZiIsMCx7InN0eWxlIjp7ImJvZHkiOnsibmFtZSI6ImRhc2hlZCJ9fX1dLFsyLDAsIlxcdGF1IiwyXSxbMiwxLCJcXHdpZGV0aWxkZXtmfSJdXQ==
\[\begin{tikzcd}
	& Z \\
	{X'} && X
	\arrow["\tau"', from=1-2, to=2-1]
	\arrow["{\widetilde{f}}", from=1-2, to=2-3]
	\arrow["f", dashed, from=2-1, to=2-3]
\end{tikzcd}\]
with \(Z\) smooth and \(\tau\) birational.
\begin{prop*}[Proposition \ref{prop_uniruled_above_branch}]
In the above notation, suppose \(f\) is a Calabi-Yau Galois-like cover. Then each essential irreducible component \(D \subset \widetilde{f}^{-1}(\Bran(\widetilde{f}))\) with its reduced structure is a uniruled variety.
\end{prop*}
\begin{proof}

By introducing some blow-up \(\sigma \colon \simY \rightarrow Y\) we obtain the diagram

% https://q.uiver.app/#q=WzAsNSxbMiwxLCJYJyJdLFs0LDEsIlgiXSxbMywwLCJaIl0sWzAsMSwiWSJdLFsxLDAsIlxcd2lkZXRpbGRle1l9Il0sWzAsMSwiZiIsMCx7InN0eWxlIjp7ImJvZHkiOnsibmFtZSI6ImRhc2hlZCJ9fX1dLFsyLDAsIlxcdGF1IiwyXSxbMiwxLCJcXHdpZGV0aWxkZXtmfSJdLFszLDAsIlxccHNpIiwwLHsic3R5bGUiOnsiYm9keSI6eyJuYW1lIjoiZGFzaGVkIn19fV0sWzQsMiwiXFx3aWRldGlsZGV7XFxwc2l9Il0sWzQsMywiXFxzaWdtYSIsMl0sWzMsMSwiXFxwaGkiLDIseyJjdXJ2ZSI6NCwic3R5bGUiOnsiYm9keSI6eyJuYW1lIjoiZGFzaGVkIn19fV0sWzQsNCwiRyIsMCx7InJhZGl1cyI6MSwiYW5nbGUiOi00NX1dLFszLDMsIkciLDAseyJyYWRpdXMiOjEsImFuZ2xlIjotNDUsInN0eWxlIjp7ImJvZHkiOnsibmFtZSI6ImRhc2hlZCJ9fX1dXQ==
\[\begin{tikzcd}
	& {\widetilde{Y}} && Z \\
	Y && {X'} && X
	\arrow["G", from=1-2, to=1-2, loop, in=105, out=165, distance=5mm]
	\arrow["{\widetilde{\psi}}", from=1-2, to=1-4]
	\arrow["\sigma"', from=1-2, to=2-1]
	\arrow["\tau"', from=1-4, to=2-3]
	\arrow["{\widetilde{f}}", from=1-4, to=2-5]
	\arrow["G", dashed, from=2-1, to=2-1, loop, in=105, out=165, distance=5mm]
	\arrow["\psi", dashed, from=2-1, to=2-3]
	\arrow["\phi"', curve={height=24pt}, dashed, from=2-1, to=2-5]
	\arrow["f", dashed, from=2-3, to=2-5]
\end{tikzcd}\]
with \(\widetilde{Y} \) smooth and \(G=\Gal(\phi)\) acting on \(\widetilde{Y}\) by automorphisms. Set \(\widetilde{\phi}:= \widetilde{f}\circ \widetilde{\psi}\). If we choose an irreducible divisor \(E\subset \widetilde{Y}\) lying above \(D\subset Z\), then \(E\) uniruled would imply \(D\) uniruled. Since \(E\subset \widetilde{\phi}^{-1}(\Bran(\widetilde{\phi}))\), we may focus on the diagram

% https://q.uiver.app/#q=WzAsMyxbMiwxLCJYIl0sWzAsMSwiWSJdLFsxLDAsIlxcd2lkZXRpbGRle1l9Il0sWzIsMSwiXFx0YXUiLDJdLFsxLDAsIlxcdmFycGhpIiwyLHsic3R5bGUiOnsiYm9keSI6eyJuYW1lIjoiZGFzaGVkIn19fV0sWzIsMCwiXFx3aWRldGlsZGV7XFx2YXJwaGl9Il1d
\[\begin{tikzcd}
	& {\widetilde{Y}} \\
	Y && X
	\arrow["\sigma"', from=1-2, to=2-1]
	\arrow["{\widetilde{\phi}}", from=1-2, to=2-3]
	\arrow["\phi"', dashed, from=2-1, to=2-3]
\end{tikzcd}\]
and reduce to proving that \(E\) is uniruled.
%Consider the Stein factorization of \(\widetilde{\phi}\),
% https://q.uiver.app/#q=WzAsMyxbMSwxLCJYIl0sWzAsMCwiXFx3aWRldGlsZGV7WX0iXSxbMiwwLCJTIl0sWzEsMCwiXFx3aWRldGlsZGV7XFxwaGl9IiwyXSxbMSwyLCJwIl0sWzIsMCwicSJdXQ==
%\[\begin{tikzcd}
%	{\widetilde{Y}} && S \\
%	& X
%	\arrow["p", from=1-1, to=1-3]
%	\arrow["{\widetilde{\phi}}"', from=1-1, to=2-2]
%	\arrow["q", from=1-3, to=2-2]
%\end{tikzcd}.\]
% If instead \(E\not \subset  \Exc(\widetilde{\phi})=\Exc(p)\), then \(E\) is birational to an irreducilbe divisor \(p(E)=E'\subset S\) and \(q(E')\subset \Bran(\widetilde{\phi})\)clearly is the quotient map under the action of \(G\), so fibers are \(G\)-orbits. In particular, points in the same orbit have the same ramification index so \(E\subset \Ram(\widetilde{\psi})\). But \(Y\) and \(X\) have trivial canonical bundle, which means \(\phi\) is étale wherever defined. We get \[\Ram(\widetilde{\phi})\subset \Exc(\sigma),\] and since each component of the exceptional locus of a birational map onto a smooth variety is well-known to be uniruled, we get that \(E\) is unriuled and we are done.
As \(Y\) and \(X\) have trivial canonical bundle, we have
\[\Ram(\widetilde{\phi})= \Exc(\sigma)\]
by (\ref{eq_Ram_=_Exc}). As every component of the exceptional locus of a birational morphism onto a smooth variety is uniruled, it suffices to show that \(E\subset \Ram(\widetilde{\phi})\).
Let \(x \in E\). If \(x\in \Exc(\widetilde{\phi})\) then \(x\) is surely a ramification point for \(\widetilde{\phi}\). If \(x\notin \Exc(\widetilde{\phi})\) then \(\widetilde{\phi}^{-1} \widetilde{\phi}(x)=G x\). But \(\widetilde{\phi}(x)\in \Bran (\widetilde{\phi})\) so there is at least one point of ramification \(y\) in \(Gx\). Since \(\widetilde{\phi}\) is \(G\)-equivariant, it then ramifies in all the fiber \(Gx\), in particular on \(x\) as well. We conclude that \(E\subset \Ram(\widetilde{\phi})\) and we are done. 
\end{proof}

Let us recall the construction of the \hk{} 8-fold of \cite{LLSvS2017}.
We start from a smooth cubic fourfold \(W\subset \proj(V)=\proj^5\) which does not contain any plane. Let \(M\) be the variety of generalized twisted cubics on \(M\), meaning the irreducible component of \(\Hilb^{3t+1}_W\) which contains the smooth twisted cubics lying on \(W\). It has dimension \(10\) and is smooth (\cite[Theorem A]{LLSvS2017}). For any curve \(C\in M\), its linear span \(\langle C\rangle \subset \proj(V)\) is a \(\proj^3\). This defines a morphism
\[M \xrightarrow{\sigma } \Gr(4,V)\]
which by \cite[Theorem B]{LLSvS2017} factors through an \(8\)-fold \(Z'\) as

% https://q.uiver.app/#q=WzAsMyxbMCwwLCJNIl0sWzEsMSwiR3IoNCxWKSJdLFswLDEsIlonIl0sWzAsMSwiXFxzaWdtYSJdLFswLDIsImEiLDJdLFsyLDEsImIiLDJdXQ==
\[\begin{tikzcd}
	M \\
	{Z'} & {Gr(4,V)}
	\arrow["a"', from=1-1, to=2-1]
	\arrow["\sigma", from=1-1, to=2-2]
	\arrow["b"', from=2-1, to=2-2]
\end{tikzcd}.\]

The \hk{} 8-fold \(Z\) of \cite{LLSvS2017} is then obtained as a divisorial contraction \(Z'\rightarrow Z\).
To define the Voisin map and its branch divisor, we only need to describe these morphisms over appropriate open subspaces.

There is an open subscheme \(U\subset Z'\) whose complement has codimension \(\geq 2\) and whose points parametrize pairs \((S,\Dcal)\) where \(S\subset W\) is a cubic surface with at worst \(A_1\) singularities and \(\Dcal\) the complete linear system of rank 2 (\cite[Theorem 1.2]{LLSvS2017}) of twisted cubics lying on \(S\). If we denote the cubic surface cut out by \(\langle C \rangle\) as \(S_C:=W\cap \langle C\rangle\), then any curve \([C] \in a^{-1}(U)\) is sent to \((S_C, |C|)\).

The Voisin map \[\nu \colon  F_1(W)\times F_1(W)\dashrightarrow Z\] 
is defined as follows:
for two skew lines \([L],[L']\in F_1(W)\), the linear span \(\langle L , L'\rangle\) intersects \(W\) in a cubic surface \(S_{L,L'}\). If \(L\) and \(L'\) are general, \(S_{L,L'}\) is smooth. We can then define \[\nu([L],[L']):=(S_{L,L'},|L-L'-K_{S_{L,L'}}|).\]
%In terms of twisted cubics, what is happening is the following: for any \(x\in L\), we may write \(\langle L' , x\rangle \cap S= L' \cup C_{L,x}\) where \(C_{L,x}\) is some residual cubic curve. By construction \(C_{L,x}\) is an element of the linear system \(|L-L'-K_S|\) for any \(x\in L\).
By \cite[Proposition 4.8]{Voi2016}, we have \(\deg \nu =6\). 

Let us make the following remark, which will be used later. 
\begin{rmk}\label{rmk_dscriminant_in_Gr}
The branch divisor \(\Bran(\nu)\subset Z\) is mapped via \(b\) to the discriminant divisor \(\mathcal{D}\subset\Gr(4,V)\), parametrizing singular cubic surfaces sections. It is indeed clear from the description in \cite{Voi2016} that \(b\circ \nu\) is étale over the open set \(\Gr(4,V)\setminus \Dcal\) parametrizing smooth cubics surfaces
\end{rmk}
We now apply Proposition \ref{prop_uniruled_above_branch} to study Question \ref{quest_Galois_factors} on Example \ref{eg_voisinmap_LLSvS}. 

\begin{cor}\label{cor_uniruled_LLSvS}
    If \(W\) is very general, the map \[\nu \colon  F_1(W)\times F_1(W)\dashrightarrow Z\] is not a Calabi-Yau Galois-like cover.
\end{cor}

\begin{proof}[Proof of Corollary \ref{cor_uniruled_LLSvS}]
Let \(\simnu \colon Y\rightarrow Z\) be a resolution of indeterminacies of \(\nu\) by the blow-up \(\tau \colon \simY \rightarrow Y\). By Proposition \ref{prop_uniruled_above_branch}, it suffices to exhibit one essential component \(D\) of \(\simnu^{-1}(\Bran(\simnu)\) which is not uniruled. We first remark that there must be an essential component \(D\subset \simnu^{-1}(\Bran(\simnu))\) such that \(D\not \subset \Ram(\simnu)\). 
Indeed, as \(W\) is very general, we can assume \(\rho_Z=1\), in which case \(\Bran(\simnu)\) is ample and so \(\simnu^* (\Bran(\simnu))\) is nef.
By the Negativity Lemma this implies \(\tau_* \simnu^* \Bran(\simnu)\) is not trivial, meaning there is some irreducible divisor \(D\subset \supp (\simnu^* \Bran(\simnu))\) which is not contracted by \(\tau\). 
By \(\Ram(\simnu)=\Exc(\tau)\), this is the same as \(D\not \subset \Ram(\simnu)\). By the projection formula on Chow groups, \(\simnu_* \simnu^* \Bran(\simnu)= \deg (\simnu) \Bran(\simnu)\), so \(\simnu_* D \) is a component of \(\Bran(\simnu)\).

Assume by contradiction that \(D\) is uniruled. The birational image \(N\) of \(D\) inside \(F_1(W)\times F_1(W)\) is then a uniruled divisor of \(F_1(W)\times F_1(W)\). Denote the projections on the first and second factors of \(F_1(W)\times F_1(W)\) as \(p_1\) and \(p_2\). If \(p_1(N)\subset F_1(W)\) is a proper closed algebraic subset, then for dimensional reasons \(p_1(N)\) must have codimension \(1\), and \(N=p_1(N)\times F_1(W)\) by irreducibility of \(N\).
If instead \(p_1(N)=F_1(W)\), the fact that \(F_1(W)\) is not uniruled implies that the general rational curve of the uniruling of \(N\) must be contracted, which implies it cannot be contracted by \(p_2\), so \(p_2(N)\subset F_1(W)\) is a proper closed algebraic subset. But then, as before, \(N = F_1(W)\times p_1(N)\). 
Up to permuting, we may thus assume \(N=p_1(N)\times F_1(W)\). Now fix one element \([L]\in p_1(N)\): by the previous argument, for any \([L']\in F_1(W)\), 
\[\nu([L],[L'])\in \Bran(\nu)\]
hence \[b\circ \nu ([L],[L']) \in \Bran(b\circ \nu)\subset \Dcal\]
by Remark \ref{rmk_dscriminant_in_Gr}.

If \(G_L\) denotes the Grassmanian parametrizing projective \(\proj^3\)'s in \(\proj(V)\) containing \(L\), then \(G_L\) has dimension 4 and the natural generically finite map
\[b\circ \nu \colon  F_1(W)\times F_1(W) \rightarrow \Gr(4,V)\]
sends \(\{[L]\}\times F_1(W)\) into \(G_L\), which therefore must be dominant onto \(G_L\). However, by Bertini, for the general element \([\Pi]\in G_L\), the corresponding \(3\)-space \(\Pi\) intersects \(W\) in a smooth cubic surface, that is, \(\Pi\not \in \Dcal\). This is a contradiction.
\end{proof}

\section{Beauville-Bogomolov negativity of branch divisors}\label{sec_exceptional_branch}

For this section we need to recall the following definition (see for example \cite[Definition 1.4]{Den2023}).
\begin{df}\label{def_exceptional}
    Let \(X\) be a \hk{} manifold and let \(q_X\) be its Beauville-Bogomolov-Fujiki form. We say that an \(\R\)-divisor \(D=\sum_i a_i D_i\) is \textit{\(q_X\)-exceptional} when the Gram matrix 
    \[(q_X(D_i,D_j))_{i,j}\]
    is negative definite.
\end{df}

\begin{rmk}\label{rmk_exceptional_components_independent}
    Note that from the above condition follows that the cohomology classes \([D_i]\) are linearly independent in \(H^2(X,\R)\). Moreover, if \(E\subset X\) is a subdivisor of \(D\) (meaning \(D-E\) is effective) and \(D\) is \(q_X\)-exceptional then \(E\) is also exceptional.
\end{rmk}

Let \(X\) be a \hk{} manifold and \(\psi \colon Y\dashrightarrow X\) a CY Galois cover. 
Set \(G:=\Gal(\phi)\) and consider the open subset \(V\subset Y\) where \(\phi\) and all elements \(g\in G\) are defined. One has that \(Y\setminus V\) has codimension at least \(2\) in \(Y\), and 
    \[\psi \colon V\rightarrow X\] is étale, inducing an open embedding
    \[\faktor{V}{G}\hookrightarrow X .\]
We denote the image of this open embedding as \(U\subset X\).

\begin{thm}\label{prop_branch_exceptional}
In the above notation, we consider the branch divisor 
\[B:=\Bran(\psi)\subset X\]
and we suppose that \(V/G\) admits a small compactification, i.e., there exists a normal, \(\Q\)-factorial projective variety \(Z\) which contains a codimension two closed subset whose complement is isomorphic to \(V/G\).
Then \(B\) is \(q_X\)-exceptional.
\end{thm}

\begin{proof}
Let \(Z\) be a small compactification of \(V/G\).
    Then the natural birational map \[f \colon X\dashrightarrow Z\]
    is a birational contraction, meaning that its inverse \(f^{-1}\) contracts no divisors. Moreover, the prime divisors \(P \subset X\) that are contracted by \(f_*\) are precisely those supported on \(X\setminus U\). 
    Indeed, for any prime divisor \(P\), if \(U'\) denotes the locus of definition of \(f\) then \(U'\cap P\neq \emptyset\) and \(f_* P\) is equal to the divisorial component of \(\overline{f_{|U'}(P\cap U')}\), which might be empty. But \(f_{|U'}(U'\setminus U)\subset Z\setminus U\) and the latter is of codimension at least 2 in \(Z\), so all the prime divisors on \(U'\) which are not contracted must intersect \(U\), and conversely it is clear that if \(P\) intersects \(U\) then it is not contracted.
    
    We introduce a resolution of indeterminacies 
    % https://q.uiver.app/#q=WzAsMyxbMCwxLCJYIl0sWzIsMSwiXFxmYWt0b3J7Vn17R30iXSxbMSwwLCJcXHNpbVgiXSxbMCwxLCJmIiwwLHsic3R5bGUiOnsiYm9keSI6eyJuYW1lIjoiZGFzaGVkIn19fV0sWzIsMSwiXFxzaW1mIl0sWzIsMCwiXFx0YXUiLDJdXQ==
\[\begin{tikzcd}
	& \simX \\
	X && {Z}
	\arrow["\tau"', from=1-2, to=2-1]
	\arrow["\simf", from=1-2, to=2-3]
	\arrow["f", dashed, from=2-1, to=2-3]
\end{tikzcd}.\]
The support of the branch divisor \(B\subset X\) is precisely the divisorial part of the branch locus, that is the divisorial part of \(X\setminus U\). But by the above argument a prime divisor \(P\subset X\) is contracted by \(f_*\) if and only if \(P\subset X\setminus U\), so the support of \(\tau^*B\) contains all strict transforms of those primitive divisors \(P\) that are contracted by \(f_*\) (they are in fact the only components of \(\tau^*B\) that are not \(\tau\)-exceptional) which means \(f\) is \textit{\(B\)-strictly negative} in the sense of \cite[Définition 2.2]{Dru2011}. We can then apply \cite[Lemme 2.5]{Dru2011} to conclude that the negative part \(N(B)\) of the Zariski decomposition of \(B\) is actually supported on each irreducible component of \(B\): 
\[|B|=|N(B)|.\]
We know prove that in fact \(B=N(B)\). Following the treatment of Zariski- Boucksom decomposition in \cite{KMPP2019} and \cite{Den2023}, the positive part \(P(B)=B-N(B)\) can be written as
\[P(B)=\sum_i a_iB_i,\]
where \(B_i\) are the irreducible Cartier divisors on which \(B\) is supported and \(a_i\geq0\). What we need to prove is \(a_i=0\) for each \(i\). By construction \(P(B)\) is \(q_X\)-nef, which implies
\[q_X(\sum_i a_i B_i)=q_X(P(B))\geq 0.\]
On the other hand, we already showed that each \(B_i\) is a component of the support of \(N(B)\); by definition of \(q_X\)-exceptional we then have that any nontrivial linear combination of  the \(B_i\) has negative Beauville-Bogomolov square, which allows us to conclude that \(a_i=0\) for each \(i\) and thus \(B=N(B)\).

\end{proof}
\begin{proof}[Proof of Theorem \ref{prop_branch_is_exceptional}]
By hypothesis \(G\) acts on \(Y\) by automorphisms. We observe that the assumptions in Theorem \ref{prop_branch_exceptional} are satisfied. Indeed, one can take \(Z=Y/G\) as a small compactification of the open subscheme \(V/G\subset X\), where as above \(V\) is the \(\Gal(f)\)-stable open subset on which \(f\) is defined. Applying Theorem \ref{prop_branch_exceptional} and Remark \ref{rmk_exceptional_components_independent}, we conclude.
    
\end{proof}

%\(\bullet\) Note that the \hk{} manifold \(X:=\Km(A)^{[2]}\) has \(\rho_X\geq 18\), as its Neron-Severi must contain an ample class, the Hilbert-Chow divisor and the Kummer lattice coming from \(\Km(A)\). 
%This makes \(X\) quite special in its moduli space: one may wonder if there can still be an abelian Galois for more general \(X\), for example ones with lower \(\rho_X\). By Theorem \ref{prop_type_over_HK4fold}, for \(n=2\) we can restrict their existence to those cases where \(\rho_X \geq 12\).
%The natural question is then to determine the minimal Picard rank \(\rho_X\in \{12,\dots 17\}\) for which \(X\) admits any abelian Galois cover.

%\(\bullet\) Try and construct \(\Km_2\)-type rational quotients of (simple?) abelian fourfolds of Weil type.    ANSWER FOUND: NOT POSSIBLE!

\color{black}
\bibliographystyle{alpha}
\bibliography{main}

\end{document}